\documentclass[12pt, amsfonts]{amsart}

\usepackage{amsmath,amssymb,amsthm,mathtools}
\usepackage{eucal}
\usepackage[top=3.5cm, bottom=3.5cm, left=2.5cm, right=2.5cm, includefoot]{geometry}
\usepackage[dvipdfmx, colorlinks=true, backref=page]{hyperref}
\usepackage[all]{xy}

\numberwithin{equation}{section}

\SelectTips{eu}{12}

\newtheorem{theorem}{Theorem}[section]
\newtheorem{proposition}[theorem]{Proposition}
\newtheorem{lemma}[theorem]{Lemma}
\newtheorem{corollary}[theorem]{Corollary}

\theoremstyle{definition}
\newtheorem{definition}[theorem]{Definition}
\newtheorem{example}[theorem]{Example}

\newtheorem{question}[theorem]{Question}
\newtheorem{assumption}[theorem]{Assumption}

\theoremstyle{remark}

\newcommand{\Z}{\mathbb{Z}}
\newcommand{\Q}{\mathbb{Q}}
\newcommand{\R}{\mathbb{R}}
\newcommand{\C}{\mathbb{C}}
\newcommand{\T}{\mathcal{T}}
\newcommand{\Ker}{\operatorname{Ker}}
\newcommand{\Coker}{\operatorname{Coker}}

\makeatletter
\newcommand{\xMapsto}[2][]{\ext@arrow 0599{\Mapstofill@}{#1}{#2}}
\def\Mapstofill@{\arrowfill@{\Mapstochar\Relbar}\Relbar\Rightarrow}
\makeatother

\title{The Stiefel-Whitney classes of moment-angle manifolds are trivial}

\author{Sho Hasui}
\address{Department of Mathematical Sciences, Osaka Prefecture University, Sakai, 599-8531, Japan}
\email{s.hasui@ms.osakafu-u.ac.jp}

\author{Daisuke Kishimoto}
\address{Faculty of Mathematics, Kyushu University, Fukuoka, 819-0395, Japan}
\email{kishimoto@math.kyushu-u.ac.jp}

\author{Akatsuki Kizu}
\address{Department of Mathematics, Kyoto University, Kyoto, 606-8502, Japan}
\email{kizu.akatsuki.85e@st.kyoto-u.ac.jp}

\subjclass[2010]{57N65, 55N91}

\keywords{moment-angle manifold, Stiefel-Whitney class, equivariant Stiefel-Whitney class, cobordism}

\begin{document}

  \maketitle

  \begin{abstract}
    We prove that the Stiefel-Whitney classes of a moment-angle manifold, not necessarily smooth, are trivial. We also consider Stiefel-Whitney classes of the partial quotient of a moment-angle manifold.
  \end{abstract}

  \baselineskip.525cm


  \parskip .05in
  \parindent .0pt
  \baselineskip.525cm


  \section{Introduction}\label{Introduction}

  Davis and Januszkiewicz \cite{DJ} introduced a space which is now called a \emph{moment-angle complex} as a topological generalization of the homogeneous coordinate for a toric variety associated to a simplicial fan \cite{Co}. We recall its definition. Let $K$ be a simplicial complex with vertex set $\{1,2,\ldots,m\}$. The moment-angle complex for $K$ is defined by
  \begin{equation}
    \label{definition Z_K}
    Z_K=\bigcup_{\sigma\in K}Z_\sigma
  \end{equation}
  where $Z_\sigma=X_1\times\cdots\times X_m$ such that $X_i=D^2$ for $i\in\sigma$ and $X_i=S^1$ for $i\not\in\sigma$. Note that the definition of $Z_K$ in the above form is due to Buchstaber and Panov \cite{BP}. The moment-angle complex $Z_K$ has been a central object of study in toric topology. In particular, when $Z_K$ is a topological manifold, not necessarily smooth, it has been studied in many contexts \cite{BLdMV,CLdM,GLdM,L,MP,PUV,T}. Some conditions on $K$ which guarantee that $Z_K$ is a topological manifold are known. Buchstaber and Panov \cite{BP} proved that $Z_K$ is a topological manifold whenever $K$ is a simplicial sphere, and later, Cai \cite{C} generalized this result to a generalized homology sphere which is a homology manifold having the same homology as a sphere.

  Suppose that $Z_K$ is a topological manifold. By definition, a torus $T^m=(S^1)^m$ acts naturally on $Z_K$, which restricts to a free action of the diagonal subgroup $S^1\cong\Delta\subset T^m$. Then $Z_K$ is the boundary of a topological manifold $Z_K\times_\Delta D^2$. So if $Z_K$ is smooth, then it is null-cobordant, or equivalently, all of its Stiefel-Whitney numbers are trivial. Hence it is natural to ask whether or not the Stiefel-Whitney classes of $Z_K$ themselves are trivial whenever $Z_K$ is a smooth manifold.

  On the other hand, for a connected topological manifold $M$, not necessarily smooth, the Wu classes $\nu_i\in H^i(M;\Z/2)$ can be defined, so that the $i$-the Stiefel-Whitney class of $M$ can be defined by
  \[
    w_i(M)=\sum_{j=0}^i\mathrm{Sq}^j\nu_{i-j}
  \]
  too. Then the Stiefel-Whitney classes of any moment-angle manifold, not necessarily smooth, can be defined. So we can generalize the above question as:

  \begin{question}
    \label{question 1}
    Are the Stiefel-Whitney classes of a moment-angle manifold, not necessarily smooth, trivial?
  \end{question}

  The aim of this paper is to answer the above question.

  \begin{theorem}
    \label{main}
    The Stiefel-Whitney classes of every moment-angle manifold, not necessarily smooth, are trivial.
  \end{theorem}

  To prove Theorem \ref{main}, we will construct in Section \ref{Equivariant Stiefel-Whitney class} equivariant Stiefel-Whitney classes for a topological manifold, not necessarily smooth. We will see in Section \ref{Proof of Theorem and its extension} that Theorem \ref{main} can be extended to a real moment-angle manifold, and in Section \ref{Quotient manifold}, we will also consider the triviality of the Stiefel-Whitney numbers of the quotient manifold $Z_K/T$, where $T$ is a subtorus of $T^m$ acting freely on $Z_K$.

  \subsection*{Acknowledgement}

  The authors are grateful to the referee for useful advice and comments. The first author was supported by JSPS KAKENHI Grant Number JP18K13414, and the second author was supported by JSPS KAKENHI Grant Number JP17K05248 and JP19K03473.


  \section{Equivariant Stiefel-Whitney class}\label{Equivariant Stiefel-Whitney class}

  This section recalls the definition of the Stiefel-Whitney classes of a topological manifold due to Fadell \cite{F}, and generalizes them to an equivariant context. Throughout this section, let $M$ denote an $R$-oriented connected $n$-dimensional topological manifold, where $R$ is a commutative ring.

  Let $\T$ be a subspace of the space of paths in $M$ consisting of constant paths and paths $\ell\colon[0,1]\to M$ such that $\ell(t)=\ell(0)$ implies $t=0$. Let $\T_0$ be the subspace of $\T$ consisting of non-constant paths. Then as in \cite[Proposition 3.8]{F}, the evaluation map
  \[
    p\colon\T\to M,\quad\ell\mapsto\ell(0)
  \]
  yields a locally trivial fibration pair
  \[
    (\T,\T_0)\to M
  \]
  with fiber homotopy equivalent to $(\R^n,\R^n-0)$. This is called the \emph{tangent fiber space}. As in \cite[Proposition 3.12]{F}, this fibration pair is equivalent to the fibration pair
  \begin{equation}
    \label{micro tangent bundle}
    (M,M-x)\to(M\times M,M\times M-\Delta)\to M
  \end{equation}
  induced from the first projection $M\times M\to M$, where $\Delta\subset M\times M$ is the diagonal set. Since \eqref{micro tangent bundle} is essentially the micro tangent bundle of $M$ in the sense of Milnor, the tangent fiber space is essentially equivalent to the micro tangent bundle of $M$.

  Clearly, $H^n(p^{-1}(U);R)$ for open sets $U$ of $M$ define a sheaf on $M$, implying we can define an orientation of $M$ over $R$ in the obvious way. So if $M$ is simply-connected or $R=\Z/2$, $M$ is orientable over $R$.

  We have the following Thom isomorphism \cite[Theorem 5.2]{F}.

  \begin{theorem}
    \label{Thom isomorphism}
    If $M$ is orientable over $R$ in the sense above, there is a cohomology class $\Phi\in H^n(\T,\T_0;R)$ such that the map
    \[
      \phi\colon H^*(M;R)\to H^{*+n}(\T,\T_0;R),\quad x\mapsto p^*(x)\smile\Phi
    \]
    is an isomorphism.
  \end{theorem}

  The cohomology class $\Phi$ is called the \emph{Thom class} of $M$. Note that the Thom class $\Phi$ is unique for $R=\Z/2$. Now we are ready to define the Stiefel-Whitney classes of $M$.

  \begin{definition}
    Let $R=\Z/2$. The $i$-th \emph{Stiefel-Whitney class} of $M$ is defined by
    \[
      w_i(M)=\phi^{-1}(\mathrm{Sq}^i\phi(1)).
    \]
  \end{definition}

  If $M$ is smooth, then by \cite[Proposition 3.17]{F}, the tangent fiber space $(\T,\T_0)\to M$ is equivalent to the tangent bundle pair $(TM,TM-M)\to M$. So the above definition of the Stiefel-Whitney classes is consistent with the usual smooth case. Moreover, consistency with the definition using the Wu classes is also proved in \cite[Theorem 6.17]{F}.

  We generalize the above definition of the Stiefel-Whitney classes to an equivariant context. So we let a topological group $G$ act on $M$ from the right. Note that $G$ acts on the space of paths $M^I$ by
  \[
    (\ell\cdot g)(t)=\ell(t)\cdot g
  \]
  for $\ell\in M^I$ and $g\in G$. Clearly, this action restricts to $\T$ and $\T_0$ such that the map $p\colon\T\to M$ is $G$-equivariant. For the rest of this section, we will make the following assumption.

  \begin{assumption}
    \label{assumption}
    The cohomology local coefficient systems over $R$ of fibrations
    \begin{align}
      \label{fibration 1}
      &M\xrightarrow{j}M\times_GEG\to BG\\
      \label{fibration 2}
      &(\T,\T_0)\xrightarrow{j}(\T,\T_0)\times_GEG\to BG
    \end{align}
    are trivial.
  \end{assumption}

  We prove the existence of the equivariant Thom class of $M$. For a $G$-space $X$, let $j\colon X\to X\times_GEG$ denote the natural inclusion. For a pair of $G$-spaces $(X,A)$, let $H^*_G(X,A;R)$ denote its equivariant cohomology, that is,
  \[
    H_G^*(X,A;R)=H^*((X,A)\times_GEG;R).
  \]

  \begin{proposition}
    \label{equivariant Thom class}
    There is a unique cohomology class $\Phi_G\in H^n_G(\T,\T_0;R)$ such that
    \[
      j^*(\Phi_G)=\Phi.
    \]
  \end{proposition}

  \begin{proof}
    Let $E_r$ denote the Serre spectral sequence for a fibration \eqref{fibration 2}. By Assumption \ref{assumption}, there is an isomorphism
    \[
      E_2^{p,q}\cong H^p(BG;H^q(\T,\T_0;R)).
    \]
    Then by Theorem \ref{Thom isomorphism}, $E_2^{p,q}=0$ for $q<n$, so $\bigoplus_{p+q=n}E_2^{p,q}=E_2^{0,n}$ and $E_\infty^{0,n}=E_2^{0,n}\cong H^n(\T,\T_0;R)$. Thus we obtain that the map $j^*\colon H_G^n(\T,\T_0;R)\to H^n(\T,\T_0;R)$ is an isomorphism. In particular, we get a unique cohomology class $\Phi_G\in H^n_G(\T,\T_0;R)$ satisfying $j^*(\Phi_G)=\Phi$, completing the proof.
  \end{proof}

  We call the cohomology class $\Phi_G$ the \emph{equivariant Thom class} of $M$. Now we prove the equivariant Thom isomorphism. As mentioned above, the map $p\colon\T\to M$ is $G$-equivariant, and so it induces a map $p_G\colon\T\times_GEG\to M\times_GEG$.

  \begin{theorem}
    \label{equivariant Thom isomorphism}
    The map
    \[
      \phi_G\colon H^*_G(M;R)\to H^{*+n}_G(\T,\T_0;R),\quad x\mapsto p_G^*(x)\smile\Phi_G
    \]
    is an isomorphism.
  \end{theorem}

  \begin{proof}
    Let $B^k$ denote the $k$-skeleton of $BG$. Let $M^k$ and $(\T^k,\T^k_0)$ denote the pullbacks of $M\times_GEG\to BG$ and $(\T,\T_0)\times_GEG\to BG$ over $B_k$. Let $\Psi$ be a cocycle representing $\Phi_G$, and let $\Psi_k=\Psi\vert_{(\T^k,\T^k_0)}$. We define a map
    \[
      \phi_k\colon C^*_G(M^k;R)\to C^{*+n}_G(\T^k,\T^k_0;R),\quad x\mapsto p_G^*(x)\smile\Psi_k
    \]
    of cochain complexes. By definition, there is a commutative diagram
    \[
      \xymatrix{
        C^*_G(M^k;R)\ar[r]^(.48){\phi_k}\ar[d]&C^{*+n}_G(\T^k,\T^k_0;R)\ar[d]\\
        C^*_G(M^{k-1};R)\ar[r]^(.44){\phi_{k-1}}&C^{*+n}_G(\T^{k-1},\T^{k-1}_0;R).
      }
    \]
    Let $E_r$ and $\widehat{E}_r$ denote the Serre spectral sequences for the fibrations
    $M\times_GEG\to BG$ and $(\T,\T_0)\times_GEG\to BG$, respectively. Then the map
    \[
      \bar{\phi}_G\colon C_G^*(M;R)\to C_G^*(\T,\T_0;R),\quad x\mapsto p_G^*(x)\smile\Psi
    \]
    induces a map of spectral sequences $f_r\colon E_r^{p,q}\to\widehat{E}_r^{p,q+n}$ (with degree shift).

    On the other hand, by assumption, there are isomorphisms
    \[
      E_2^{p,q}\cong H^p(BG;H^q(M;R))\quad\text{and}\quad
      \widehat{E}_2^{p,q}\cong H^p(BG;H^q(\T,\T_0;R)).
    \]
    By the construction of these isomorphisms, the map $f_2\colon E_2^{p,q}\to\widehat{E}_2^{p,q+n}$ is identified with the map
    \[
      H^p(BG;H^q(M;R))\to H^p(BG;H^{q+n}(\T,\T_0;R)).
    \]
    induced from $\phi\colon H^q(M;R)\to H^{q+n}(\T,\T_0;R)$. Then by Theorem \ref{Thom isomorphism}, the map $f_\infty\colon E_\infty\to\widehat{E}_\infty$ is an isomorphism, implying the map $\phi_G\colon H^*_G(M;R)\to H^{*+n}_G(\T,\T_0;R)$ is an isomorphism too. Thus the proof is finished.
  \end{proof}

  We are ready to define the equivariant Stiefel-Whitney classes of $M$.

  \begin{definition}
    Let $R=\Z/2$. The $i$-th \emph{equivariant Stiefel-Whitney class} of $M$ is defined by
    \[
      w_i^G(M)=\phi_G^{-1}(\mathrm{Sq}^i\phi_G(1)).
    \]
  \end{definition}

  We will use the following property.

  \begin{proposition}
    \label{SW}
    $j^*(w_i^G(M))=w_i(M)$.
  \end{proposition}

  \begin{proof}
    There is a commutative diagram
    \[
      \xymatrix{
        H_G^*(M;R)\ar[r]^(.4){\phi_G}\ar[d]_{j^*}&H_G^{*+n}(\T,\T_0;R)\ar[d]^{j^*}\\
        H^*(M;R)\ar[r]^(.4){\phi}&H^{*+n}(\T,\T_0;R).
      }
    \]
    Indeed, for each $x\in H_G^*(M;R)$, we have
    \[
      j^*(\phi_G(x))=j^*(p_G^*(x)\smile\Phi_G)=j^*(p_G^*(x))\smile j^*(\Phi_G)=p(j^*(x))\smile\Phi=\phi(j^*(x)).
    \]
    Then we get
    \begin{align*}
      j^*(w_i^G(M))&=j^*(\phi_G^{-1}(\mathrm{Sq}^i\phi_G(1)))=\phi^{-1}(j^*(\mathrm{Sq}^i\phi_G(1)))=\phi^{-1}(\mathrm{Sq}^ij^*(\phi_G(1)))\\
      &=\phi^{-1}(\mathrm{Sq}^i\phi(j^*(1)))=\phi^{-1}(\mathrm{Sq}^i\phi(1))=w_i(M).
    \end{align*}
    Thus the proof is done.
  \end{proof}


  \section{Proof of Theorem \ref{main} and its extension}\label{Proof of Theorem and its extension}

  This section proves Theorem \ref{main} and its extension to a real moment-angle manifold. Let $K$ be a simplicial complex with $m$ vertices, and let $T^m=(S^1)^m$. Then $T^m$ acts naturally on $Z_K$. Suppose that $Z_K$ is a topological manifold. Then $Z_K$ is connected. Since $BT^m$ is simply-connected, the moment-angle manifold $Z_K$ with an action of $T^m$ satisfies Assumption \ref{assumption}. Then in particular, the equivariant Stiefel-Whitney classes $w_i^{T^m}(Z_K)$ can be defined.

  Now we are ready to prove Theorem \ref{main}. Let $DJ_K$ denote the Borel construction $Z_K\times_{T^m}ET^m$.

  \begin{proof}
    [Proof of Theorem \ref{main}]
    Consider a fibration $Z_K\xrightarrow{j}DJ_K\to BT^m$. As in \cite[Theorem 4.8]{DJ}, the map $DJ_K\to BT^m$ is surjective in cohomology, so by the standard spectral sequence argument, we can see that the map $j\colon Z_K\to DJ_K$ is trivial in cohomology. Thus the proof is complete by Proposition \ref{SW}.
  \end{proof}

  We extend Theorem \ref{main} to a real moment-angle manifold. Let $K$ be a simplicial complex with vertex set $\{1,2,\ldots,m\}$. The real momend-angle complex for $K$ is defined by
  \[
    \R Z_K=\bigcup_{\sigma\in K}\R Z_\sigma
  \]
  where $\R Z_\sigma=X_1\times\cdots\times X_m$ such that $X_i=D^1$ for $i\in\sigma$ and $X_i=S^0$ for $i\not\in\sigma$. As well as a moment-angle manifold, we call $\R Z_K$ a real moment-angle manifold if it is a topological manifold. Cai \cite[Theorem 2.3]{C} proved that $\R Z_K$ is a topological manifold whenever $K$ is a simplicial sphere. Note that there is a natural action of a 2-torus $(\Z/2)^m$ on $\R Z_K$.

  \begin{lemma}
    \label{RZ_K action}
    The following group actions are trivial in mod 2 cohomology:
    \begin{enumerate}
      \item the $(\Z/2)^m$-action on $\R Z_K$;

      \item the $(\Z/2)^m$-action on $(\T,\T_0)$ over $\R Z_K$ whenever $\R Z_K$ is a topological manifold.
    \end{enumerate}
  \end{lemma}

  \begin{proof}
    (1) By \cite[Proposition 3.3]{C}, we can see that any mod 2 homology class of $\R Z_K$ admits a representative which is fixed by the action of $(\Z/2)^m$. So the action of $(\Z/2)^m$ on $\R Z_K$ is trivial in mod 2 homology, so it is trivial in mod 2 cohomology too.

    \noindent(2) By Theorem \ref{Thom isomorphism}, $H^n(\T,\T_0;\Z/2)\cong\Z/2$. Then the action of $(\Z/2)^m$ on $H^n(\T,\T_0;\Z/2)$ is trivial, so in particular, the Thom class $\Phi$ is fixed by the action of $(\Z/2)^m$. By Theorem \ref{Thom isomorphism}, each element of $H^*(\T,\T_0;\Z/2)$ is of the form $p^*(x)\smile\Phi$ for some $x\in H^*(\R Z_K;\Z/2)$. Then since $p\colon(\T,\T_0)\to\R Z_K$ is $(\Z/2)^m$-equivariant and the action of $(\Z/2)^m$ on $\R Z_K$ is trivial in mod 2 cohomology, for any $g\in(\Z/2)^m$, we have
    \[
      (p^*(x)\smile\Phi)\cdot g=(p^*(x)\cdot g)\smile(\Phi\cdot g)=p^*(x\cdot g)\smile\Phi=p^*(x)\smile\Phi.
    \]
    Thus the action of $(\Z/2)^m$ on $(\T,\T_0)$ is trivial in mod 2 cohomology, completing the proof.
  \end{proof}

  By Lemma \ref{RZ_K action}, the action of $(\Z/2)^m$ on $\R Z_K$ satisfies Assumption \ref{assumption} whenever $\R Z_K$ is a topological manifold. Then we can define its equivariant Stiefel-Whitney classes. Now we are ready to prove an extension of Theorem \ref{main} to a real-moment-angle manifold.

  \begin{theorem}
    The Stiefel-Whitney classes of a real moment-angle manifold, not necessarily smooth, are trivial.
  \end{theorem}

  \begin{proof}
    By \cite[Theorem 4.8]{DJ}, the map $\R Z_K\times_{(\Z/2)^m}E(\Z/2)^m\to B(\Z/2)^m$ is surjective in mod 2 cohomology, so the map $\R Z_K\to \R Z_K\times_{(\Z/2)^m}E(\Z/2)^m$ is trivial in mod 2 cohomology as in the proof of Theorem \ref{main}. Thus the proof is complete by Proposition \ref{SW}.
  \end{proof}


  \section{Quotient manifold}\label{Quotient manifold}

  There are important manifolds given by the quotient of a moment-angle manifold $Z_K$ by a free action of a subtorus $T\subset T^m$, where $m$ is the number of vertices of a simplicial complex $K$. A typical example is a quasitoric manifold \cite{BP}, and there are other classes of such manifolds \cite{BP,Fr,P}. Recall that our starting point of study is the fact that every moment-angle manifold $Z_K$ is null-cobordant. So it is natural to ask whether or not the quotient manifold $Z_K/T$ is null-cobordant. This section gives an answer to this question by sorting with the dimension of $T$.

  Throughout this section, let $K$ be a simplicial complex of dimension $n-1$ having the vertex set $\{1,2,\ldots,m\}$, and let $T$ be a subtorus of $T^m$ acting freely on $Z_K$. If $K$ has no simplex, we set $n=0$. Clearly, we must have $\dim T\le m-n$. First, we examine the case $\dim T=m-n$, which includes quasitoric manifolds.

  \begin{example}
    Consider the case that $K$ is the boundary of an $n$-simplex, where $m=n+1$ in this case. The complex projective space $\C P^n$ is the quotient of $Z_K$ by a free action of a subtorus $T$ of $T^n$ with $\dim T=1$. By \cite[Theorem 1.1]{S}, $\C P^n$ is null-cobordant if only if $n$ is odd. Then $\C P^n$ is not always null-cobordant. It is well known that the total Stiefel-Whitney class of $\C P^n$ is given by
    \[
      (1+u)^{n+1}
    \]
    where $u$ is a generator of $H^2(\C P^n;\Z/2)\cong\Z/2$. So the Stiefel-Whitney classes of $\C P^n$ are trivial if and only if $n=2^k-1$ for some $k$. Then the Stiefel-Whitney classes of $\C P^n$ are not trivial, even if it is null-cobordant.
  \end{example}

Thus for $\dim T=m-n$, $Z_K/T$ being null-cobordant and the triviality of its Stiefel-Whitney classes depend on each $Z_K/T$. Next, we consider the case $\dim T<m-n$ and prove:

\begin{theorem}
  \label{SW number}
  If $\dim T<m-n$ and $Z_K/T$ is a topological manifold, then all Stiefel-Whitney numbers of $Z_K/T$ vanish.
\end{theorem}

The following corollary is immediate from Theorem \ref{SW number}.

\begin{corollary}
  \label{null cobordant}
  If $\dim T<m-n$ and $Z_K/T$ is a smooth manifold, then $Z_K/T$ is null-cobordant.
\end{corollary}

If $T$ is not maximal among subtori of $T^m$ acting freely on a moment-angle manifold $Z_K$, then $S^1$ acts freely on the quotient manifold $Z_K/T$. So as well as $Z_K$, we can see that $Z_K/T$ is null-cobordant. Then Corollary \ref{null cobordant} makes sense only when $T$ is a maximal subtorus of $T^m$ acting freely on $Z_K$. We will give below such an example.

We will take three steps to prove Theorem \ref{SW number}. The first step is to extend the free action of $T$ on $Z_K$ to an almost free action of $(m-n)$-dimensional subtorus of $T^m$ containing $T$. Recall that a simplicial complex $K$ is said to be \emph{pure} if all of its maximal simplices have the same dimension.

\begin{lemma}
  \label{pure}
  If $Z_K/T$ is a topological manifold, then $K$ is pure.
\end{lemma}

\begin{proof}
  By \cite[Theorem 1.7.19]{Pal}, there is a principal fiber bundle $T\to Z_K\to Z_K/T$. So if $Z_K/T$ is a topological manifold, then $Z_K$ is a topological manifold too. Let $\sigma$ be a maximal simplex of $K$, and let $U_i=\mathrm{Int}(D^2)$ for $i\in\sigma$ and $U_i=S^1-\{x_0\}$ for $i\not\in\sigma$, where $x_0\in S^1$. Then $U_1\times\cdots\times U_m$ is an open set of $Z_K$, which is homeomorphic with $\R^{m+|\sigma|}$. Thus if $Z_K$ is a topological manifold, then $K$ must be pure, completing the proof.
\end{proof}

We introduce a rational characteristic matrix of $K$. An integer matrix $(\lambda_1,\ldots,\lambda_m)$ for $\lambda_1,\ldots,\lambda_m\in\Z^n$ is called a \emph{rational characteristic matrix} of $K$ if the column vectors $\lambda_{i_1},\ldots,\lambda_{i_k}$ are linearly independent over $\Q$ for each simplex $\{i_1,\ldots,i_k\}\in K$. For an $n\times m$ matrix $A=(a_1,\ldots,a_m)$ and a subset $\sigma=\{i_1<\cdots<i_k\}\subset\{1,\ldots,m\}$, let
\[
  A(\sigma)=(a_{i_1},\ldots,a_{i_k}).
\]

\begin{lemma}
  \label{complement}
  Suppose that $K$ is pure and integer vectors $\lambda_1,\ldots,\lambda_n,\theta_1,\ldots,\theta_{m-n}\in\Z^m$ satisfy the following conditions:

  \begin{enumerate}
    \item $\lambda_1,\ldots,\lambda_n,\theta_1,\ldots,\theta_{m-n}$ are linearly independent over $\Q$;

    \item $\lambda_i$ and $\theta_j$ are orthogonal for each $i=1,\ldots,n$ and $j=1,\ldots,m-n$.
  \end{enumerate}
  Then $\Lambda={}^t(\lambda_1,\ldots,\lambda_n)$ is a rational characteristic matrix of $K$ if and only if $\det\Theta(\bar{\sigma})\ne 0$ for each maximal simplex $\sigma\in K$, where $\Theta={}^t(\theta_1,\ldots,\theta_{m-n})$ and $\bar{\sigma}$ denotes the complement of $\sigma$ in $\{1,\ldots,m\}$.
\end{lemma}

\begin{proof}
  Since $K$ is pure, $\Lambda$ is a rational characteristic matrix of $K$ if and only if $\det\Lambda(\sigma)\ne 0$ for each maximal simplex $\sigma\in K$. Now we suppose that $\det\Lambda(\sigma)\ne 0$ and $\det\Theta(\bar{\sigma})=0$. We may assume $\sigma=\{1,\ldots,n\}$. Let
  \[
    \lambda_i=\begin{pmatrix}\lambda_i(1)\\\lambda_i(2)\end{pmatrix}\quad\text{and}\quad\theta_j=\begin{pmatrix}\theta_j(1)\\\theta_j(2)\end{pmatrix}
  \]
  for $\lambda_i(1),\theta_j(1)\in\Z^n$ and $\lambda_i(2),\theta_j(2)\in\Z^{m-n}$. Then
  \[
    \Lambda(\sigma)=(\lambda_1(1),\ldots,\lambda_n(1))\quad\text{and}\quad\Theta(\bar{\sigma})=(\theta_1(2),\ldots,\theta_{m-n}(2)).
  \]
  So since $\det\Theta(\bar{\sigma})=0$, vectors $\theta_1(2),\ldots,\theta_{m-n}(2)$ are linearly dependent over $\Q$, and so there are $c_1,\ldots,c_{m-n}\in\Q$ such that $c_1\theta_1(2)+\cdots+c_{m-n}\theta_{m-n}(2)=0$ and $(c_1,\ldots,c_{m-n})\ne 0$. Let $x=c_1\theta_1(1)+\cdots+c_{m-n}\theta_{m-n}(1)$. Then by the first condition, we have $x\ne 0$. Since $\det\Lambda(\sigma)\ne 0$, $\lambda_1(1),\ldots,\lambda_n(1)$ are linearly independent over $\Q$, so there are $d_1,\ldots,d_n\in\Q$ such that $x=d_1\lambda_1(1)+\cdots+d_n\lambda_n(1)$. Let $y=d_1\lambda_1(2)+\cdots+d_n\lambda_n(2)$. Then we get
  \[
    d_1\lambda_1+\cdots+d_n\lambda_n=\begin{pmatrix}x\\y\end{pmatrix}\quad\text{and}\quad c_1\theta_1+\cdots+c_{m-n}\theta_{m-n}=\begin{pmatrix}x\\0\end{pmatrix}.
  \]
  Since $x\ne 0$, these two vectors are not orthogonal, which contradicts to the second condition, and thus it cannot occur that $\det\Lambda(\sigma)\ne 0$ and $\det\Theta(\bar{\sigma})=0$ simultaneously. Quite similarly, we can show that it cannot occur that $\det\Lambda(\sigma)=0$ and $\det\Theta(\bar{\sigma})\ne 0$ simultaneously too. Thus the proof is complete.
\end{proof}

Every $n\times m$ integer matrix $\Lambda$ defines a homomorphism $T^m\to T^n$, and we denote the identity component of its kernel by $T(\Lambda)$. Then $T(\Lambda)$ is an $(m-n)$-dimensional torus.

\begin{lemma}
  \label{extension}
  If $K$ is pure, then there is a rational characteristic matrix $\Lambda$ of $K$ such that $T\subset T(\Lambda)$.
\end{lemma}

\begin{proof}
  Let $\Lambda_1$ be a $k\times m$ integer matrix whose row vectors generate $T$, where $k=\dim T$. By Lemma \ref{complement}, it is sufficient to find an $(m-n-k)\times m$ matrix $\Lambda_2$ with entries in $\Q$ such that
  \[
    \det\begin{pmatrix}\Lambda_1\\\Lambda_2\end{pmatrix}(\bar{\sigma})\ne 0
  \]
  for each maximal simplex $\sigma\in K$. For a maximal simplex $\sigma\in K$, let $U(\sigma)$ be the subspace of $M(m-n-k,m;\Q)$, the space of $(m-n-k)\times m$ rational matrices, consisting of matrices $\Lambda_2$ satisfying the above condition. Then $U(\sigma)$ is non-empty and Zariski open. Thus the intersection of $U(\sigma)$ for all maximal simplices $\sigma\in K$ is non-empty, completing the proof.
\end{proof}

Let $S(\Lambda)=T^m/T(\Lambda)$, so that $S(\Lambda)$ acts on $Z_K/T(\Lambda)$. The second step is to compute the $S(\Lambda)$-equivariant rational cohomology of $Z_K/T(\Lambda)$.

\begin{lemma}
  \label{free action}
  Let $G$ be a topological group acting on a space $X$. If the restriction of the $G$-action to a normal subgroup $H$ is free, then the natural map
  \[
    X\times_GEG\to (X/H)\times_{G/H}E(G/H)
  \]
  is a weak homotopy equivalence.
\end{lemma}

\begin{proof}
  There is a commutative diagram
  \[
    \xymatrix{
      G\ar[r]\ar[d]&X\ar[r]\ar[d]&X\times_GEG\ar[d]\\
      G/H\ar[r]&X/H\ar[r]&(X/H)\times_{G/H}E(G/H)
    }
  \]
  in which each row is a homotopy fibration. Let $F$ be the homotopy fiber of the right vertical map. Then by taking the homotopy fibers of all vertical maps, we get a homotopy fibration $H\to H\to F$, where the first map is the identity map. Thus $F$ is weakly contractible, completing the proof.
\end{proof}

\begin{lemma}
  \label{H iso}
  Let $\Lambda$ be a rational characteristic matrix of $K$ as in Lemma \ref{extension}. Then the natural map
  \[
    H^*_{S(\Lambda)}(Z_K/T(\Lambda);\Q)\to H^*_{T^m}(Z_K;\Q)
  \]
  is an isomorphism.
\end{lemma}

\begin{proof}
  First, we prove that the map $Z_\sigma\times_{T^m}ET^m\to (Z_\sigma/T(\Lambda))\times_{S(\Lambda)}ES(\Lambda)$ is an isomorphism in rational cohomology, where $Z_\sigma$ is as in \eqref{definition Z_K}. Let $X=X_1\times\cdots\times X_m$ such that $X_i=0$ for $i\in\sigma$ and $X_i=S^1$ for $i\not\in\sigma$, where we regard $S^1$ is the unit sphere in $\C$. Then $X$ is a $T^m$-equivariant deformation retract of $Z_\sigma$, so that there is a commutative diagram
  \[
    \xymatrix{
      X\times_{T^m}ET^m\ar[r]\ar[d]_\simeq&(X/T(\Lambda))\times_{S(\Lambda)}ES(\Lambda)\ar[d]^\simeq\\
      Z_\sigma\times_{T^m}ET^m\ar[r]&(Z_\sigma/T(\Lambda))\times_{S(\Lambda)}ES(\Lambda).
    }
  \]
  Hence we aim to prove that the top map is an isomorphism in rational cohomology. Since $\Lambda$ is a rational characteristic matrix of $K$, $T(\Lambda)$ acts uniformly on $X$ such that the isotropy subgroup $G$ is finite. Then $T(\Lambda)/G$ acts freely on $X$, and so by Lemma \ref{free action}, the natural map
  \[
    X\times_{T^m/G}E(T^m/G)\to(X/T(\Lambda))\times_{S(\Lambda)}ES(\Lambda)
  \]
  is a weak homotopy equivalence. On the other hand, the homotopy fiber of the map $X\times_{T^m}ET^m\to X\times_{T^m/G}E(T^m/G)$ is homotopy equivalent to $BG$ which is rationally contractible. Then we obtain that the map $Z_\sigma\times_{T^m}ET^m\to(Z_\sigma/T(\Lambda))\times_{S(\Lambda)}ES(\Lambda)$ is an isomorphism in rational cohomology.

  Next, we induct on the number of simplices of $K$, fixing $\Lambda$, where $K$ may have ghost vertices. If $K$ has no simplex and consists only of ghost vertices, then $Z_K=T^m$, and so by Lemma \ref{free action}, the statement holds. Let $\tau<\sigma$ be simplices of $K$. Since $(Z_\sigma,Z_\tau)$ is a $T^m$-equivariant NDR pair,  $(Z_\sigma\times_{T^m}ET^m,Z_\tau\times_{T^m}ET^m)$ is an NDR pair. Quite similarly, we see that $((Z_\sigma/T(\Lambda))\times_{S(\Lambda)}ES(\Lambda),(Z_\tau/T(\Lambda))\times_{S(\Lambda)}ES(\Lambda))$ is an NDR pair too. Then for each maximal simplex $\sigma$ of $K$, we can apply the Mayer-Vietotris sequene to get a commutative diagram
  \[
    \xymatrix{
      \cdots\ar[r]&H^*_{S(\Lambda)}(Z_K/T(\Lambda))\ar[r]\ar[d]&H^*_{S(\Lambda)}(Z_L/T(\Lambda))\oplus H^*_{S(\Lambda)}(Z_{\sigma}/T(\Lambda))\ar[r]\ar[d]&H^*_{S(\Lambda)}(Z_{\partial\sigma}/T(\Lambda))\ar[r]\ar[d]&\cdots\\
      \cdots\ar[r]&H^*_{T^m}(Z_K)\ar[r]&H^*_{T^m}(Z_L)\oplus H^*_{T^m}(Z_{\sigma})\ar[r]&H^*_{T^m}(Z_{\partial\sigma})\ar[r]&\cdots
    }
  \]
  with exact rows, where $L=K-\sigma$, we consider ghost vertices for $Z_\sigma$ and $Z_{\partial\sigma}$, and we omit the coefficient $\Q$. Thus by the five lemma, the induction proceeds, and the proof is complete.
\end{proof}

The final step is to prove the vanishing in the equivariant cohomology.

\begin{lemma}
  \label{vanishing}
  Let $U$ be any subtorus of $T^m$. Then for $*>m+n-\dim U$,
  \[
    H^*(Z_K/U;\Q)=0.
  \]
\end{lemma}

\begin{proof}
We prove the lemma by induction on the number of simplices of $K$, where $K$ may have ghost vertices. First, if $K$ has no simplex, then $n=0$ and $Z_K/U=(S^1)^m/U$ is homeomorphic with $T^{m-\dim U}$, so the statement holds. Next, suppose that $K$ has a simplex, and let $\sigma\in K$ be a maximal simplex. Let $L=K\setminus\{\sigma\}$. Consider the long exact sequence
\[
  \cdots\to H^*(Z_K/U,Z_L/U;\Q)\to H^*(Z_K/U;\Q)\to H^*(Z_L/U;\Q)\to\cdots.
\]
By the induction hypothesis, $H^*(Z_L/U;\Q)=0$ for $*>m+n-\dim U$. Then we aim to show that $H^*(Z_K/U,Z_L/U;\Q)=0$ for $*>m+n-\dim U$. Since $(Z_K,Z_L)$ is a $T^m$-equivariant NDR pair, $(Z_K/U,Z_L/U)$ is an NDR pair. Then, by the excision property, there is an isomorphism
\[
  H^*(Z_K/U,Z_L/U;\Q)\cong H^*(Z_\sigma/U,Z_{\partial\sigma}/U;\Q),
\]
where $Z_\sigma$ is as in \eqref{definition Z_K}. Then we show $H^*(Z_\sigma/U,Z_{\partial\sigma}/U;\Q)=0$ for $*>m+n-\dim U$.

Let $X$ be as in the proof of Lemma \ref{H iso}. Then $Z_\sigma/U\simeq X/U$. Since $T^m$ acts on $X\cong (S^1)^{m-|\sigma|}$ through the projection $\rho\colon T^m\to T_\sigma$, where $T_\sigma=\{(z_1,\ldots,z_m)\in T^m\mid z_i=1\text{ for }i\not\in\sigma\}$. Then there is a homeomorphism
\[
  X/U\cong\Coker\rho|_U.
\]
Since $\Ker\rho|_U\subseteq\Ker\rho$, we have $\dim\Ker\rho|_U\le \dim\Ker\rho=|\sigma|$, implying $\dim\rho(U)\ge \dim U-|\sigma|$. Thus we see that $X/U$ is homeomorphic with a torus of dimension $\leq (m-|\sigma|)-(\dim U-|\sigma|)=m-\dim U$. In particular, we see that $H^*(Z_\sigma/U;\Q)=0$ for $*>m-\dim U$. Now we consider the exact sequence
\[
  \cdots\to H^*(Z_\sigma/U,Z_{\partial\sigma}/U;\Q)\to H^*(Z_\sigma/U;\Q)\to H^*(Z_{\partial\sigma}/U;\Q) \to\cdots.
\]
By the induction hypothesis, $H^*(Z_{\partial\sigma}/U;\Q)=0$ for $*>m+n-1-\dim U$. Thus we obtain $H^*(Z_\sigma/U,Z_{\partial\sigma}/U;\Q)=0$ for $*>m+n-\dim U$, completing the proof.
\end{proof}

Let $d=m+n-\dim T$ and $S=T^m/T$. Then $S$ acts on $Z_K/T$.

\begin{proposition}
  \label{H^d}
  If $\dim T<m-n$ and $Z_K$ is a topological manifold, then the natural map
  \[
    H^d_{S}(Z_K/T)\to H^d(Z_K/T)
  \]
  is trivial.
\end{proposition}

\begin{proof}
  There is a commutative diagram
  \[
    \xymatrix{
      H^d_{S(\Lambda)}(Z_K/T(\Lambda);\Q)\ar[r]\ar[d]&H^d(Z_K/T(\Lambda);\Q)\ar[d]\\
      H^d_{S}(Z_K/T;\Q)\ar[r]&H^d(Z_K/T;\Q).
    }
  \]
  By Lemmas \ref{free action} and \ref{H iso}, the left vertical map is an isomorphism. By Lemma \ref{vanishing}, we also have $H^d(Z_K/T(\Lambda);\Q)=0$. Thus we obtain that the bottom map is trivial, completing the proof.
\end{proof}

Now we are ready to prove Theorem \ref{SW number}.

\begin{proof}
  [Proof of Theorem \ref{SW number}]
  By assumption, $Z_K/T$ is a smooth manifold, so that we need to show all of its Stiefel numbers vanish. Let $j\colon Z_K/T\to ES\times_S(Z_K/T)$ denote the inclusion. By Proposition \ref{SW}, $j^*(w_i^S(Z_K/T))=w_i(Z_K/T)$. Then by Proposition \ref{H^d},
  \[
    w_{i_1}(Z_K/T)\cdots w_{i_k}(Z_K/T)=j^*(w_{i_1}^S(Z_K/T)\cdots w_{i_k}^S(Z_K/T))=0
  \]
  for $i_1+\cdots+i_k=d$, completing the proof.
\end{proof}

As mentioned above, Corollary \ref{null cobordant} makes sense only when $T$ is a maximal subtorus of $T^m$ acting freely on $Z_K$. We give an example of such a maximal subtorus $T$. For $t_1<\cdots<t_m$, let $P_i=(t_i,t_i^2,\ldots,t_i^n)\in\R^n$. Recall that the \emph{cyclic polytope} $C_n(m)$ is defined by the convex hull of points $P_1,\ldots,P_m$ in $\R^n$. It is well known that the combinatorial type of $C_n(m)$ is independent from the choice of $t_1<\cdots<t_m$. The cyclic polytope $C_n(m)$ is a simplicial $n$-dimensional polytope with $m$ vertices \cite[Theorem 0.7]{Z}, and so $Z_{\partial C_n(m)}$ is a smooth manifold such that the action of $T^m$ is smooth \cite[Lemma 3.1.2]{BP}. Hence if a subtorus $T$ of $T^m$ acts freely on $Z_{\partial C_n(m)}$, then $Z_{\partial C_n(m)}/T$ is a smooth manifold.

  \begin{proposition}
    \label{cyclic polytope}
    There is a 2-dimensional subtorus $T$ of $T^9$ which is maximal among subtori acting freely on $Z_{\partial C_6(9)}$.
  \end{proposition}

  \begin{proof}
    It is shown in \cite{H} that there is no 3-dimensional subtorus of $T^9$ acting freely on $Z_{\partial C_6(9)}$. So we only need to give a 2-dimensional subtorus of $T^9$ acting freely on $Z_{\partial C_6(9)}$.

    Let $T$ be the image of a map $T^2\to T^9$ defined by a matrix:
    \begin{equation}
      \label{matrix}
      \begin{pmatrix}
        1&0&1&0&1&0&1&0&1\\
        0&1&0&1&0&1&0&1&1
      \end{pmatrix}
    \end{equation}
    To see that $T$ acts freely on $Z_{\partial C_6(9)}$, it is sufficient to show that $T$ acts freely on
    \[
      D(\sigma)=\{(z_1,\ldots,z_9)\in Z_{\partial C_6(9)}\mid z_i=0\text{ for }P_i\in\sigma\}
    \]
    for each facet $\sigma$ of $\partial C_6(9)$. Let $\{P_a,P_b,P_c\}=\{P_1,\ldots,P_9\}-\sigma$ for a facet $\sigma$ of $\partial C_6(9)$, where $a<b<c$. By Gale's evenness condition \cite[Theorem 0.7]{Z}, $b-a$ must be odd, so that we can easily verify $\det(\rho_a,\rho_b)=\pm 1$, where $\rho_i$ denotes the $i$-th column of the matrix \eqref{matrix}. Then if we remove column vectors corresponding to vertices of a facet $\sigma$ from the matrix \eqref{matrix}, then we get a $2\times 3$ matrix of rank $2$. This implies that $T$ acts freely on $D(\sigma)$, completing the proof.
  \end{proof}

  By Proposition \ref{cyclic polytope}, we see that Corollary \ref{null cobordant} is not trivial. Now we can further ask whether or not the Stiefel-Whitney classes of $Z_K/T$ are trivial, as well as Question \ref{question 1}. Here is the answer to this question.

  \begin{proposition}
    Let $T$ be the 2-dimensional subtorus of $T^9$ in Proposition \ref{cyclic polytope}. Then
    \[
      w_2(Z_{\partial C_6(9)}/T)\ne 0.
    \]
  \end{proposition}

  \begin{proof}
    By definition, the 2-dimensional subtorus $T$ of $T^9$ is the kernel of a map $T^9\to T^7$ defined by a matrix:
    \begin{equation}
      \label{matrix 2}
      \begin{pmatrix}
        -1&0&1&0&0&0&0&0&0\\
        0&-1&0&1&0&0&0&0&0\\
        -1&0&0&0&1&0&0&0&0\\
        0&-1&0&0&0&1&0&0&0\\
        -1&0&0&0&0&0&1&0&0\\
        0&-1&0&0&0&0&0&1&0\\
        -1&-1&0&0&0&0&0&0&1
      \end{pmatrix}
    \end{equation}
    By \cite[Theorem 3.7]{MT}, $T^9\cong T\times T^9/T$, and therefore we can regard $ET^9=ET\times E(T^9/T)$. Thus we see the following immediately.
    \[
      Z_{\partial C_6(9)}/T\times_{T^9/T}E(T^9/T)\simeq Z_{\partial C_6(9)}\times_{T^9}ET^9=DJ_{\partial C_6(9).}
    \]
    Then there is a homotopy commutative diagram
    \[
      \xymatrix{
        T\ar[r]\ar[d]&\ast\ar[r]\ar[d]&(\C P^\infty)^2\ar[d]\\
        Z_{\partial C_6(9)}\ar[r]\ar[d]&DJ_{\partial C_6(9)}\ar[r]^\alpha\ar@{=}[d]&(\C P^\infty)^9\ar[d]^\lambda\\
        Z_{\partial C_6(9)}/T\ar[r]&DJ_{\partial C_6(9)}\ar[r]&(\C P^\infty)^7
      }
    \]
    where all columns and rows are homotopy fibrations and the map $\lambda$ is defined by a matrix \eqref{matrix 2}. Consider the Serre spectral sequence for the middle row homotopy fibration. Then since $Z_{\partial C_6(9)}$ is simply-connected \cite[Corollary 3.3.4]{BP}, the map $\alpha$ is an isomorphism in $H^2$. By the left fibration, we can also see that $Z_{\partial C_6(9)}/T$ is simply-connected. Then by applying the Serre exact sequence to the bottom fibration, we get
    \[
      H^2(Z_{\partial C_6(9)}/T;\Z)=\Z\langle v_1,\ldots,v_9\rangle/\Z\langle\theta_1,\ldots,\theta_7\rangle,\quad\theta_i=\lambda_{i1}v_1+\cdots+\lambda_{i9}v_9
    \]
    where $\lambda_{ij}$ denotes the $(i,j)$-entry of a matrix \eqref{matrix 2} and $\Z\langle a_1,\ldots,a_n\rangle$ means a free abelian group spanned by $a_1,\ldots,a_n$. On the other hand, quite similarly to \cite[Corollary 6.7]{DJ}, we can see that
    \[
      w_2(Z_{\partial C_6(9)}/T)=[v_1+\cdots+v_9]\in H^2(Z_{\partial C_6(9)}/T;\Z/2).
    \]
    Thus we obtain that $H^2(Z_{\partial C_6(9)}/T;\Z/2)$ has a basis $[v_1],[v_2]$ and $w_2(Z_{\partial C_6(9)}/T)=[v_1]+[v_2]$, implying $w_2(Z_{\partial C_6(9)}/T)\ne 0$. Therefore the proof is complete.
  \end{proof}

\end{document}